\documentclass[12pt, oneside]{amsart}
\usepackage{geometry}
\usepackage{graphicx}
\usepackage{amssymb}
\allowdisplaybreaks

\RequirePackage[OT1]{fontenc}
\RequirePackage{amsthm,amsmath,amsfonts}
\RequirePackage[numbers]{natbib}
\RequirePackage[colorlinks,citecolor=blue,urlcolor=blue]{hyperref}
\makeatletter
\newcommand*\bigcdot{\mathpalette\bigcdot@{.5}}
\newcommand*\bigcdot@[2]{\mathbin{\vcenter{\hbox{\scalebox{#2}{$\m@th#1\bullet$}}}}}
\makeatother

\newtheorem{theorem}{Theorem}[section]
   \newtheorem{lemma}[theorem]{Lemma}
   \newtheorem{proposition}[theorem]{Proposition}
   \newtheorem{example}[theorem]{Example}
   \newtheorem{definition}[theorem]{Definition}
   
   \newtheorem{corollary}[theorem]{Corollary}
   \newtheorem{remark}[theorem]{Remark}
   
\newtheorem{com}[theorem]{Comments}
    \newcommand \is{\bigcdot}
    \newcommand \id{1\!\!1}
    \newcommand\graph[1]{[\! [ #1]\! ]}
    
    \newcommand\cro[1]{\langle #1\rangle}
    \newcommand\sgn[1]{\textrm{sgn}(#1)}
    
    \newcommand\wh{\widehat}
    \newcommand\wt{\widetilde}
    \newcommand\F{{\mathcal F}}
    \newcommand\G{{\mathcal G}}
    
    \newcommand\ff{{\mathbb F}}
    
    \renewcommand\gg{{\mathbb G}}
    \renewcommand\P{\mathbb P}
    \newcommand\Q{\mathbb Q}
    \newcommand\R{\mathbb R}
    \newcommand \E {{\mathbb E}}

    \newcommand{\Rbrack}{[\![}
    \newcommand{\Lbrack}{]\!]}
    
    \newcommand\ooo[1]{\!\!\!\phantom{I}^o #1}
\newcommand\p[1]{\!\!\!\phantom{I}^p #1}
\newcommand\oo[1]{#1^o}
\newcommand\pp[1]{#1^p}

\begin{document}


\title[Thin times and random times' decomposition]{Thin times and random times' decomposition}

\author[Anna Aksamit, Tahir Choulli and Monique Jeanblanc]{Anna Aksamit, Tahir Choulli and Monique Jeanblanc}

\address{Anna Aksamit\\
School of Mathematics and Statistics\\
University of Sydney\\
Carslaw F07\\
2006, Sydney\\
Australia}
\email{anna.aksamit@sydney.edu.au}

\address{Tahir Choulli\\
Mathematical and Statistical Sciences Department\\
University of Alberta\\
632 Central Academic Building\\
Edmonton AB T6G 2G1\\
Canada}
\email{tchoulli@ualberta.ca}

\address{Monique Jeanblanc\\
Laboratoire de Math\'ematiques et Mod\'elisation d'\'Evry (LaMME), UMR CNRS 8071\\
Universit\'e d'\'Evry-Val-d'Essonne\\
23 Boulevard de France\\
91037 \'Evry cedex\\
France}
\email{monique.jeanblanc@univ-evry.fr}

\thanks{AA and MJ wish to acknowledge the generous financial supports of {\it 'Chaire Markets in transition'},
 French Banking Federation  and ILB, Labex
ANR 11-LABX-0019.   AA wishes to acknowledge the support of the
European Research Council under the European Union's Seventh
Framework Programme (FP7/2007-2013) / ERC grant agreement no.
335421.
The research of TC is supported by the Natural Sciences and Engineering
Research Council of Canada (NSERC), through Grant RGPIN 04987.}


\date{\today}

\maketitle


\begin{abstract}
The paper studies thin  times which are
random times whose graph is contained in a countable union of the
graphs of stopping times with respect to a reference filtration
$\ff$.  We show that a generic random time can be decomposed
into thin and thick parts, where the second is a random time
avoiding all $\ff$-stopping times.
Then, for a given random time $\tau$, we introduce $\ff^\tau$, the
smallest right-continuous filtration containing $\ff$ and making
$\tau$ a stopping time, and we show that, for a thin time $\tau$,
each $\ff$-martingale is an $\ff^\tau$-semimartingale, i.e., the
hypothesis $({\mathcal H}^\prime)$ for $(\ff, \ff^\tau)$ holds. We
present applications to honest times, which can be seen as last
passage times, showing classes of filtrations which can only
support thin honest times, or can accommodate thick honest times as
well.
\end{abstract}


\section{Introduction}

The paper studies the class of thin times in an enlargement
of filtration framework. The concept naturally fits, and
complements, the studies of random times and progressive
enlargement of filtrations. A random time defined on a filtered
probability space $(\Omega,\mathcal G,\mathbb F,\mathbb P)$ with
$\ff=(\F_t)_{t\geq 0}$, is  a random variable with values in
$[0, \infty]$. In the literature of enlargement on filtration, e.g. Mansuy and Yor \cite{mansuyyor} and Nikeghbali \cite{ni:eg},  it
is common to assume that the random time $\tau$ avoids all
$\ff$-stopping times, i.e., $\P(\tau=T<\infty)=0$ for any
$\ff$-stopping time $T$. The motivation
behind our work is to explore what happens if this condition
fails.
{In Definition \ref{thin_defi} we introduce
\emph{thin times} which satisfy the   opposite property, i.e.,
their graph is contained in a countable union of graphs of $\ff$-stopping times. 
The given name is motivated by the fact that the graph of a thin
random time is contained in a thin set (see \cite[Chapter I, Definition
1.30]{JSlimit} for definition and main properties of thin sets). 
The notion of
thin time was mentioned, but not developed,
for the first time in Dellacherie and Meyer \cite{arlequine} under the name \emph{arlequine random variable} 
referring to the costume of the Harlequin which is made of patches of different colors.
On the other hand we also work with \emph{thick times} which are introduced in Definition \ref{thick_defi}}, and satisfy the above avoidance condition, i.e.,
 and the graph of a thick
random time does not intersect any thin set (i.e., the intersection
is an evanescent set).
In Section \ref{s:thin}, we show the first results on thin times. 
Our study strongly relies on
the notion of dual optional projection and other processes linked
to the general theory of stochastic processes, in particular, to
the enlargement of filtration theory.

Since their introduction in the 1980's, enlargements of
filtrations have remained an important tool and field of study in
the theory of stochastic processes. In fact the theory has seen
its second youth recently with revised interest sparked by
applications in mathematical finance. These include,
in particular, credit risk and modelling of asymmetry of information, where one considers a financial market where different agents have different levels of information.

Enlargement of filtration theory, to which we contribute here,
focuses on the properties of stochastic processes under a change
of filtration. The behaviour of (semi)martingales under a suitable
change of filtration may be seen as parallel to absolutely
continuous change of measure and Girsanov's theorem (see
\cite{jacod1985grossissement,song:th,yoeurp}). 
It is of a fundamental interest to provide new classes of enlargements under which the semimartingale property is stable.

Thin times form a new class of random times which possesses this property under progressive enlargement.
Recall that for a random time $\tau$,
$\ff^\tau:=(\F^{\tau}_t)_{t\geq 0}$ denotes the filtration $\ff$
progressively enlarged with $\tau$, and is given by
\begin{equation}
\label{Ftau}
\F^{\tau}_t:=\bigcap_{s>t}\left ( \F_s \lor \sigma(\tau\land s) \right)\quad \textrm{for any $t\geq 0$}.
\end{equation}
The fundamental question in the enlargement of filtration theory
is if all $\ff$-martingales remain 
$\ff^\tau$-semi\-martin\-gales. If the latter property is
satisfied we say that the hypothesis $({\mathcal H}^\prime)$ holds
for $(\ff, \ff^\tau)$, in
which case we are inte\-rested in the 
$\ff^\tau$-semimartingale decomposition of $\ff$-martingales. The
main result in Section \ref{s:thinH} is Theorem \ref{decomposition1}
where we establish the hypothesis $({\mathcal H}^\prime)$ for thin
random times and give a corresponding semimartingale
decomposition. It extends a previous result by Jeulin \cite[Lemma
(4,11)]{j}  which deals with countably valued random times. Instead we choose countably many $\ff$-stopping times which are already
captured in the reference filtration $\ff$.
Jeanblanc and Le Cam \cite{jeanblanc2009progressive}  have established  that the hypothesis $({\mathcal H}^\prime)$ holds for a   progressive enlargement by an initial time, i.e., a random time which satisfies absolute continuity hypothesis  introduced by Jacod \cite{jacod1985grossissement} (see also \cite[Theorem 3,2]{j} and
\cite{meyer1978theoreme}).
Conceptually, we may see Theorem \ref{decomposition1} as a different way of adopting results of Jacod to progressive setting and obtaining qualitatively different results.


In Section \ref{s:decomp} we define the
decom\-posi\-tion of a random time into thick and thin parts which
we call the thin-thick decomposition.
The thin-thick decomposition is
congruent with the decomposition of a stopping time into accessible and totally
inaccessible parts.
One of the main results in this section, Theorem \ref{A0decom}, says that any random time $\tau$
admits a unique thin-thick decomposition and characterizes its thin and thick
components in terms of the dual optional projection of the
indicator process $\id_{\Rbrack \tau, \infty\Rbrack}$.
In Section \ref{s:decomp} we also show the significance of thin-thick decomposition for the hypothesis $(\mathcal H^\prime)$ and immersion in the context of the progressive
enlargement of filtration.

In Section \ref{s:honest} we turn to honest times which constitute an important and well studied class of random times (see Barlow \cite{barlow78} and Jeulin \cite{j}) and can be suitably represented as last passage times.
Adopting the notion of jumping filtration from Jacod and Skorokhod \cite{jac} we show {in Theorem \ref{jump_honest}, which is the main result of this section,} that such a filtration can only support honest times which are thin.
That includes the compound Poisson process filtration.
In \cite{jac} the link between jumping filtration and finite variation martingales is established; further developments related to purely discontinuous martingale filtrations are presented in Hannig \cite{hannig2003filtrations}.
In Theorem \ref{jump_honest} we also show that there exists a thick time in the filtration which can accommodate a non-constant continuous martingale.
In Section \ref{s:honest} we also discuss two examples of thin honest times: the last passage time at a barrier $a$ of a compound Poisson process and an example based on an approximation of a Brownian local time.

\section{A definition and some properties of thin times}
\label{s:thin}

Let $(\Omega,\mathcal G,\mathbb F,\mathbb P)$ be a filtered
probability space, where $\mathbb{F}:=(\F_t)_{t\geq 0}$ denotes a
filtration satisfying the usual conditions, and such that
$\F_\infty:=\bigcup_{t>0}\F_t \subset \mathcal G$.
For any
c\`adl\`ag process $X$ we   denote by $X_-$ the left-continuous
version of $X$, by $\Delta X$ the jump of $X$ and by $X_\infty$
the limit $\lim_{t\to \infty} X_t$ if it exists. The process $X$
is said to be increasing if, for almost all $\omega$, it satisfies
$X_t(\omega)\geq X_s(\omega)$ for all $t\geq s$. A random variable
is said to be positive if it has values in $[0, \infty)$. We
denote by $G\is X$ the stochastic integral of a predictable
process $G$ w.r.t. a semimartingale $X$, when this integral is well defined.

Consider a random time $\tau$, i.e., a $[0, \infty]$-valued
$\G$-measurable random variable. Note that a random time $\tau$ is
not necessarily $\F_\infty$-measurable. For a random time $\tau$,
we denote by $\graph{\tau}:=\{(\omega, t)\subset \Omega\times
\R^+: \tau(\omega)=t\}$ its graph. Let us recall, following
\cite{j}, some useful processes associated with  the pair
$(\ff,\tau)$. For the process
$A:=\id_{\Rbrack\tau,\infty\Rbrack}$, we denote by $A^p$ its
$\ff$-dual predictable projection and by $A^o$ its $\ff$-dual
optional projection (see Appendix \ref{projections}). By  an abuse
of language, $A^o$ is also called the dual optional projection of
the random time $\tau$. We also define two $\ff$-supermartingales
$Z$ and $\wt Z$ as the optional projections of processes $1-A$
and $1-A_-$ respectively, i.e.,
\begin{equation*}
Z_t:=\; ^o \Big [\id_{\Rbrack 0, \tau\Rbrack}\Big ]_t=\P(\tau>t|\F_t)\quad \textrm{and}\quad
\wt Z_t:=\; ^o \Big [\id_{\Rbrack 0, \tau\Lbrack}\Big ]_t=\P(\tau\geq t|\F_t).
\end{equation*}
Since the dual optional projection $A^o$ will play a crucial role
in the paper, we recall two equalities where it appears (see
\cite[Chapitre IV, section 1]{j}):
\begin{equation}
\label{martm}
A^o=m-Z \quad \textrm{and} \quad \Delta A^o=\wt Z-Z \, ,
\end{equation}
 where $m$ is a BMO $\ff$-martingale. Furthermore,
  $\wt Z=Z_-+\Delta m$.

 The following definition
contains the leading idea of the paper. It introduces a class of random times using a criterion based on $\ff$-stopping
times w.r.t. a reference filtration.

\begin{definition}
\label{thin_defi}
A random time $\tau$ is called
an $\ff$-thin time if its graph $\graph {\tau}$ is contained
in an $\ff$-thin set, i.e., if there exists a sequence of $\ff$-stopping
times $(T_n)_{n=1}^\infty$ with disjoint graphs such that $\graph
{\tau}\subset\bigcup_{n=1}^\infty\graph {T_n}$.

Let $T_0:=\infty$. We say that the sequence
$(T_n)_{n\geq 0}$ exhausts the $\ff$-thin time $\tau$ or that $(T_n)_{n\geq 0}$
is an $\ff$-exhausting sequence of the $\ff$-thin time $\tau$.

We say that the family of sets $(C_n)_{n\geq 0}$, given by $C_0:=\{\tau=\infty\}$ and $C_n:=\{\tau=T_n<\infty\}$ for $n\geq 1$,
is an $\ff$-partition of the $\ff$-thin time $\tau$.

We say that the family of bounded c\`adl\`ag $\ff$-martingales $(z^n)_{n\geq 0}$ given by its terminal values $\P(C_n|\F_\infty)$, namely
$z^n_t:=\P(C_n|\F_t)$, is a martingale family of the thin time $\tau$.
\end{definition}

If this is clear from the context we shall simply say that $\tau$ is a thin time instead of saying that $\tau$ is an $\ff$-thin time etc.
Note that a thin time $\tau$ is built from
$\ff$-stopping times, i.e., $\tau=\sum_{n\geq 0} T_n\id_{C_n}$ where $(T_n)_{n\geq 0}$ is one exhausting sequence and $(C_n)_{n\geq 0}$ is its partition.
On the other hand, given a sequence $(T_n)_{n\geq 0}$ of $\ff$-stopping times with disjoint graphs such that $T_0=\infty$ and a partition $(C_n)_{n\geq 0}$ of $\Omega$, the random time $\tau$ defined as $\tau:=\infty\id_{C_0}+\sum_{n\geq 1} T_n\id_{C_n}$ is thin.

Let us also remark that an exhausting sequence $(T_n)_{n\geq 0}$ of a thin time is not unique, however the properties of a thin time do not depend on the specific choice of an exhausting sequence.
The following proposition combines two exhausting sequence a given thin time.
\begin{proposition}
Let $\tau$ be a thin time with an exhausting sequence $(T_n)_{n\geq 0}$ and a partition $(C_n)_{n\geq 0}$. Suppose that $(S_n)_{n\geq 0}$ and $(B_n)_{n\geq 0}$ are as well an exhausting sequence and a partition of $\tau$.
Then, $\left (U_0, (U_{n,m})_{n\geq 1, m\geq 1}\right)$, defined as $U_0:=\infty$ and  $U_{n,m}:=T_n\id_{\{T_n=S_m\}}+\infty\id_{\{T_n\neq S_m\}}$ for $n\geq 1$ and $m\geq 1$, is an exhausting sequence of $\tau$ and $\left (D_0,(D_{n,m})_{n\geq 1, m\geq 1}\right)$ defined as $D_0:=\{\tau=\infty\}$ and $D_{n,m}:=C_n\cap B_m$ for $n\geq 1$ and $m\geq 1$, is the corresponding partition of $\tau$.
\end{proposition}
\begin{proof}
Firstly note that $U_{n,m}$ is a stopping time for any pair $n\geq 1$ and $m\geq 1$ since $\{T_n=S_m\}\in\F_{T_n\land S_m}$.
Secondly note that the following identity holds:
$$\tau
=\infty\id_{\{\tau=\infty\}}+\sum_{n\geq 1}\sum_{m\geq 1} T_n\id_{\{\tau=T_n=S_m<\infty\}}=\infty \id_{D_{0}}+\sum_{n\geq 1}\sum_{m\geq 1} U_{n,m}\id_{\{\tau=U_{n,m}<\infty\}}.$$
Hence it remains to show that $(U_{n,m})_{n\geq 1, m\geq 1}$ have disjoint graphs which follows by observing that the sets
$$\graph{U_{n,m}}\cap\graph{U_{k,l}}\subset\graph{T_{n}}\cap\graph{T_{k}}\quad \textrm{and}\quad
\graph{U_{n,m}}\cap\graph{U_{k,l}}\subset\graph{S_{m}}\cap\graph{S_{l}}$$
are evanescent if $n\neq k$ or $m\neq l$.
\end{proof}

Thin times, unlike other classes of random times, possess many stability properties as described in the following remark.

\begin{remark}
\label{r:stability}
(a) Let $\Q$ be absolutely continuous w.r.t. $\P$ and $\wt \ff$ be the filtration $\ff$ completed with $\Q$-null sets.
Then an $\ff$-thin time is an $\wt\ff$-thin time since $\ff\subset \wt\ff$. In other words thin times are invariant w.r.t. an absolutely continuous change of measure.\\
(b) Let $\gg$ be such that $\ff\subset \gg$. Then any $\ff$-thin time is a $\gg$-thin time since any $\ff$-stopping time is a $\gg$-stopping time. In other words, thin times are stable under filtration enlargement.\\
(c) Let $\tau$ and $\sigma$ be two $\ff$-thin times with exhausting sequences $(T_n)_{n\geq 0}$ and $(S_n)_{n\geq 0}$ respectively. Then $\tau\land\sigma$ and $\tau\lor \sigma$ are also $\ff$-thin times since
$$\graph{\tau\land\sigma}\subset \bigcup_{n\geq 1}\graph{T_n}\cup \bigcup_{n\geq 1}\graph{S_n}\quad \textrm{and}\quad \graph{\tau\lor\sigma}\subset \bigcup_{n\geq 1}\graph{T_n}\cup \bigcup_{n\geq 1}\graph{S_n}.$$
\end{remark}

The following theorem provides a useful characterization of thin time based on its $\ff$-dual optional projection.
\begin{theorem}
\label{A0jump}
A random time is a thin time if and only if its dual optional projection is a pure jump process.
\end{theorem}
\begin{proof}
For  any sequence $(S_n)_{n\geq 1}$  of $\ff$-stopping times with
disjoint graphs, we have
$$\sum_{n=1}^\infty\P(\tau=S_n<\infty)=\sum_{n=1}^\infty \E\left [\Delta A^o_{S_n}\id_{\{S_n<\infty\}}\right].$$
Since  by definition of the dual optional projection
$\E[A^o_\infty]=\P(\tau<\infty)$, and using the fact that $A^o$ is
an increasing process, we conclude that the sequence $(T_n)_{n\geq 0}$ with $T_0=\infty$ is
an exhausting sequence of  $\tau$, i.e.,  satisfies the condition
$\sum_{n=1}^\infty\P(\tau=T_n<\infty)=\P(\tau<\infty)$, if and only if it
satisfies the condition $\E[A^o_\infty]=\sum_{n=1}^\infty \E[\Delta A^o_{T_n}\id_{\{T_n<\infty\}}]$. In other words, $\tau$ is a thin time if and only if $A^o$ is a pure jump process.
\end{proof}

The next result gives  the supermartingales $Z$ and $\wt Z$ of a
thin time and their decompositions into an $\ff$-martingale
$m$ and  the increasing process $A^o$ in
terms of an exhausting sequence and martingale family of $\tau$.
This will be useful to check certain properties of thin (honest) times (we refer the reader to Sections \ref{s:thinimm} and \ref{s:honest}).

\begin{proposition}
\label{positive}
Let $\tau$ be a thin time with exhausting sequence $(T_n)_{n\geq 0}$, partition $(C_n)_{n\geq 0}$ and martingale family $(z^n)_{n\geq 0}$.
Then:
\begin{enumerate}
\item[(a)] $z^n>0$ and $z^n_->0$ a.s. on $C_n$ for each $n\geq 0$,
\item[(b)] $1-Z_\tau>0$ a.s. on $\{\tau<\infty\}$,
\item[(c)] $\wt Z_t=\sum_{n=0}^\infty \id_{\{t\leq T_n\}}z^n_t,$ $Z_t=\sum_{n=0}^\infty \id_{\{t<T_n\}}z^n_t,$
$A^o_t=\sum_{n=1}^\infty \id_{\{t\geq T_n\}}z^n_{T_n}$ and
$m_t=\sum_{n=0}^\infty z^n_{t\land T_n}.$
\end{enumerate}
\end{proposition}

\begin{proof}
(a) Define, for any $n\geq 0$, the $\ff$-stopping time
\begin{equation}
\label{Rn}
R^n:=\inf\{t\geq 0 : z^n_t=0\}.
\end{equation}
As $z^n$ is a positive c\`adl\`ag martingale, by \cite[Proposition (3.4) p.70]{revuz}, it vanishes on $\Rbrack R^n,\infty\Rbrack$.
Since $z^n$ is bounded, $z^n_\infty$ exists and:
$$\{R^n<\infty\}=\left\{\inf_{t\geq 0} z^n_t =0\right\}=\{z^n_\infty=0\}.$$
Moreover, the equality
$0=\E[z^n_\infty\id_{\{z^n_\infty=0\}}]=\E[\id_{C_n}\id_{\{z^n_\infty=0\}}]
$ implies that
$C_n\cap \{z^n_\infty=0\}$ is a null set, so as well $C_n \cap \{\inf_t z^n_t=0\}$ is a null set.
We obtain that $z^n>0$ and $z^n_->0$ a.s. on $C_n$.\\
(b) We have $Z_\tau\id_{\{\tau<\infty\}}=\sum_{n=1}^\infty \id_{C_n}Z_{T_n}$ and, on $\{T_n<\infty\}$, we have
\begin{align*}
1-Z_{T_n}&=\P(\tau \leq T_n|\F_{T_n})\geq
\P(\tau=T_n|\F_{T_n})=z^n_{T_n}.
\end{align*}
From part (a), this implies that $1-Z_\tau>0$ a.s. on $\{\tau<\infty\}$.\\
(c)
Deriving the form of $Z$ and $\wt Z$ is straightforward. To compute $A^o$, note that for any $\ff$-optional  process $X$, one has, setting
$H^n= \id_{\Rbrack T_n,\infty\Rbrack}$,
\begin{eqnarray*}
{\E\left[\int_{[0,\infty)} X_s dA^{o }_s\right]}&=&\E\left[\int_{[0,\infty)} X_sdA  _s\right]  =
\E[X_\tau\id_{\{\tau<\infty\}}]
= \sum_{n=1}^\infty \E\left[ X_{T_n}z_{T_n}^n\id_{\{T_n<\infty\}}\right]\\
&=& \sum_{n=1}^\infty
\E\left [\int_{[0,\infty)} X_sz_s^n dH^n_s\right]
 \,.\end{eqnarray*}
The form of $m$ follows by \eqref{martm}.
\end{proof}

 The following result describes how,
after a thin time, the conditional expectations with
respect to elements of $\ff^\tau$ can be expressed in terms of the
conditional expectations with respect to elements of $\ff$.
For an arbitrary random time, one is able to express $\ff^\tau$-conditional expectations in terms of $\ff$-conditional expectations only strictly before $\tau$ (this result is often referred to as key lemma in enlargement of filtration literature, see Lemma 3.1 in  \cite{elliott} and Section 3.1.1 in \cite{bjr:intro}).
A powerful property of thin times is that one can obtain this kind of result also after $\tau$ as described below. It is crucial for results in Section \ref{s:thinH}.

\begin{lemma}
\label{key}
Let $\tau$ be a thin time with exhausting sequence $(T_n)_{n\geq 0}$, partition $(C_n)_{n\geq 0}$ and martingale family $(z^n)_{n\geq 0}$. Then:
\begin{enumerate}
\item[(a)] The progressive enlargement of filtration $\ff$ with $\tau$, $\ff^\tau:=(\F^\tau_t)_{t\geq 0}$, defined in \eqref{Ftau} by $\F^\tau_t:=\bigcap_{u>t}\F_u\lor\sigma(\{\tau\leq s\}: s\leq u\}$, satisfies
\begin{equation*}
\label{eq:defFt}
\F^\tau_t
=\bigcap_{u>t}\F_u\lor \sigma (C_n\cap \{T_n\leq s\}, \; s\leq u,
\; n \geq 1).
\end{equation*}
\item[(b)] For any $n\geq 1$ and any $\G$-measurable integrable random variable $X$, we have
\[
\E\left[X|\F^\tau_t\right]\id_{\{t\geq T_n\}\cap C_n}
=\id_{\{t\geq T_n\}\cap C_n}\frac{\E\left[X \id_{
C_n}|\F_t\right]}{z^n_t}.
\]
\end{enumerate}
\end{lemma}
\begin{proof}
(a) The proof is based on monotone class theorem and we focus on a generator. The inclusion $\bigcap_{u>t}\F_u\lor \sigma (C_n\cap \{T_n\leq s\}, \; s\leq u,
\; n \geq 1)\subset \F^\tau_t$ follows since $\tau$ and $T_n$ are $\ff^\tau$-stopping times, therefore $\{\tau=T_n<\infty\}\in \F^\tau_{T_n}$ and $\{\tau=T_n<\infty\}\cap\{T_n\leq s\}\in \F^\tau_{s}$. The reverse inclusion is due to $\{\tau\leq s\}=\bigcup_{n=1}^\infty C_n\cap\{T_n\leq s\}$.

(b) By (a) and the monotone class theorem, for each
$G\in \F^\tau_t$ there exists $F\in \F_t$ such that, for any $n\geq 1$,
\begin{equation}
\label{FG} G\cap \{T_n\leq t\}\cap C_n=F\cap \{T_n\leq t\}\cap
C_n.\end{equation}
Then, using the fact that $T_n$ are $\ff$-stopping times, we have to show that
\[
\E\left[X\id_{\{t\geq T_n\}\cap C_n}z^n_t|\F^\tau_t\right]
=\id_{\{t\geq T_n\}\cap C_n}\E\left[X \id_{\{t\geq T_n\}\cap
C_n}|\F_t\right].
\]
For any $G\in \F^\tau_t$, we choose $F\in \F_t$ satisfying \eqref{FG}, and we obtain
\begin{align*}
\E\left[X\id_{\{t\geq T_n\}\cap C_n\cap G}\; z^n_t\right]
&=\E\left[X\id_{\{t\geq T_n\}\cap C_n\cap F} \; \E\left[\id_{C_n}|\F_t\right]\right]\\
&=\E\left[\id_{\{t\geq T_n\}\cap F} \; \E\left[\id_{C_n}|\F_t\right]\E\left[X\id_{C_n}|\F_t\right]\right]\\
&=\E\left[\id_{\{t\geq T_n\}\cap C_n \cap F}\; \E\left[X \id_{C_n}|\F_t\right]\right]\\
&=\E\left[\id_{\{t\geq T_n\}\cap C_n\cap G}\; \E\left[X
\id_{C_n}|\F_t\right]\right]
\end{align*}
which ends the proof, taking into account that $z^n_t>0$ on $C_n$.
\end{proof}

\begin{corollary}
It follows immediately that, for $s\leq t$,
 $n\geq 1$ and any $\G$-measurable integrable random variable $X$
 \[
\E\left[X|\F^\tau_t\right]\id_{\{s\geq T_n\}\cap C_n}
=\id_{\{s\geq T_n\}\cap C_n}\frac{\E\left[X \id_{
C_n}|\F_t\right]}{z^n_t}.
\]
\end{corollary}

\section{The hypothesis $({\mathcal H}^\prime)$ for thin times}
\label{s:thinH}

One of the vital questions in the enlargement of filtration theory is whether all semimartingales in the reference filtration remain semimartingales in an enlarged filtration, i.e., whether the hypothesis $(\mathcal H^\prime)$ holds.
In progressive enlargement setting there are only few classes of random times with this property, i.e., honest times and random times satisfying Jacod's absolutely continuous condition.
In this section we prove the hypothesis $(\mathcal H^\prime)$ for the new class of random times i.e., thin times. This contributes to and completes the existing theory in a relevant way.
Let us first recall the equivalent characterisations of the hypothesis $({\mathcal H}^\prime)$ (see
\cite[page 12]{j}).
\begin{definition}
\label{Hp}
{Let $\ff$ and $\gg$ be two filtrations such that $\ff\subset \gg$.}
Then, the hypothesis $({\mathcal H}^\prime)$ holds for $(\ff, \gg)$ if any of the following equivalent conditions holds:\\
(a) any $\ff$-semimartingale is a $\gg$-semimartingale;\\
(b) any $\ff$-martingale is a $\gg$-semimartingale;\\
(c) any bounded $\ff$-martingale is a $\gg$-semimartingale.
\end{definition}

Before formulating the result of this section we  recall a vital result by Jacod
 (see \cite[Theorem 3,2]{j} and \cite{meyer1978theoreme}) on the hypothesis $({\mathcal H}^\prime)$ in  the case of
 initial enlargement with an atomic $\sigma$-field.\\
Let $\ff^{ \, \mathcal C}$ denote the initial enlargement of the
filtration $\ff$ with the atomic $\sigma$-field
$\mathcal C:=\sigma(C_n, n\geq 0)$ generated by a partition $(C_n, n\geq 0)$ of a thin time $\tau$, i.e.,
\begin{equation}
\label{FC}
\F^{\mathcal C}_t:=\bigcap_{s>t}\F_s\lor \sigma(C_n, n\geq 0).
\end{equation}
For this case
of enlargement, Jacod's result says that the hypothesis {\bf
$({\mathcal H^\prime})$} holds for $(\ff,\ff^{ \, \mathcal C})$
and  the decomposition of any $\ff$-martingale $X$ as an $\ff^{ \,
\mathcal C}$-semimartingale is
\begin{equation}
\label{jacod_result}
X_t=\widehat X_t+\sum_n \id_{C_n} \int_0^t \frac{1}{z^n_{s-}}d\cro{X, z^n}_s,
\end{equation}
where $\widehat X$ is an $\ff^{ \, \mathcal C}$-local martingale
and $(z^n)_{n\geq 0}$ is a martingale family of the thin time $\tau$, and the predictable bracket is
computed in $\ff$.

\begin{theorem}
\label{decomposition1} Let $\tau$ be a thin time. Then
$\ff\subset \ff^\tau \subset \ff^{ \, \mathcal C}$ and the
hypothesis $({\mathcal H}^\prime)$ is satisfied for  $(\ff,
\ff^\tau)$. Moreover, for each $\ff^\tau$-predictable and bounded
process $G$ and each $\ff$-local martingale $Y$ the stochastic
integral $X:=G\bigcdot Y$ exists and is an $\ff^\tau$-semimartingale with
canonical decomposition
\begin{equation}
\label{thinH} X_t= \widehat X_t + \int_0^{t\land
\tau}\frac{G_s}{Z_{s-}}d\cro{Y,m}_s +\sum_{n=1}^\infty \id_{C_n} \int_0^t
\id_{\{s>T_n\}}\frac{G_s}{z^n_{s-}}d\cro{Y,z^n}_s,
\end{equation}
where $\widehat X$ is an $\ff^\tau$-local martingale, and the
predictable brackets are computed in $\ff$.
\end{theorem}
\begin{proof}
The fact that the hypothesis $(\mathcal H^\prime)$ holds for
$(\ff, \ff^\tau)$ follows from   \eqref{jacod_result}
and Stricker's Theorem \cite[Theorem 4, Chapter II]{protter} since
$\ff^\tau \subset \ff^{ \, \mathcal C}$. To prove the
decomposition result, let $H$ be an $\ff^\tau$-predictable bounded
process. Then, \cite[Lemma (4,4)]{j} implies that
\[
H_t=\id_{\{t\leq \tau\}}J_t+\id_{\{\tau<t\}}K_t(\tau), \quad t \geq 0,
\]
where $J$ is an $\ff$-predictable bounded process and $K: \mathbb
R_{+} \times \Omega \times\mathbb R_{+} \rightarrow \mathbb R $ is
$\mathcal P \otimes \mathcal B(\mathbb R_{+})$-measurable and
bounded. Since $\tau$ is a thin time, we can rewrite the
process $H$ as
\begin{align*}
H_t=J_t\id_{\{t\leq \tau\}}+\sum_{n=1}^\infty \id_{\{T_n<t\}}K_t(T_n)\id_{C_n}.
\end{align*}
Note that, since $\{t\leq \tau\}\subset
\{Z_{t-}>0\}$, $J$ can be chosen to satisfy
$J_t=J_t\id_{\{Z_{t-}>0\}}$ and, since $C_n\subset
\{z^n_{t-}>0\}$,  each process $K^n_t:=\id_{\{T_n<t\}}K_t(T_n)$
being $\ff$-predictable and bounded,  $K^n$ can be chosen to
satisfy $K^n_t=K^n_t\id_{\{z^n_{t-}>0\}}$.

We denote by $H^1(\ff)$ the space of 
$\ff$-local martingales $N$ s.t. {$\E([N]_\infty^{1/2})<\infty$.}
Let $N$ be an $H^1(\ff)$-martingale.
Then the stochastic integrals $J\bigcdot N$ and $K^n\bigcdot N$ are well defined and each of them is an $H^1(\ff)$-martingale.
For each $n\geq 0$ and for each bounded $\ff$-martingale $N$, by integration by parts, we have that
\begin{equation}
\label{duality}
\E\left[\id_{C_n}N_\infty\right]=\E\left[[z^n, N]_\infty\right]=\E\left[\cro{z^n, N}_\infty\right].
\end{equation}
Since $N\to \E[\id_{C_n}N_\infty]$ is a linear form, the duality $(H^1,BMO)$ implies that \eqref{duality} holds for any $H^1(\ff)$-martingale $N$.
Similarly, since $m$, given in \eqref{martm}, is a $BMO(\ff)$-martingale, for any $H^1(\ff)$-martingale $N$, the process $\cro{N,m}$ exists and we have $$\E[N_\tau]=\E\left [[N, m]_\infty\right]=\E\left[\cro{N, m}_\infty\right].$$
Therefore
\begin{align*}
\E\left [\int_0^\infty H_sd N_s\right ]=&
\E\left [\int_0^\tau J_s d N_s\right ]
+\sum_{n=1}^\infty\E\left [\id_{C_n}\int_{0}^\infty K^n_s dN_s\right]\\
=&\E\left [\int_0^\infty J_s d \cro{m,N}_s\right ]
+\sum_{n=1}^\infty\E\left [\int_{0}^\infty K^n_s d\cro{z^n,N}_s\right].
\end{align*}
Then, since for any predictable finite variation process $V$,
$\E[\int_0^\infty h_sdV_s]=\E[\int_0^\infty \;^p h_sdV_s]$, and $Z_-=\,^p (1-A_-)$, we
deduce, taking care on the specific choice of $J$ and $K$,
\begin{eqnarray*}
\E\left [\int_0^\infty H_sd N_s\right ]&=&
\E\left [\int_0^\infty \frac{Z_{s-}}{Z_{s-}}\id_{\{Z_{s-}>0\}}J_s d \cro{m,N}_s\right ]\\
&&\quad+\sum_{n=1}^\infty\E\left [\int_{0}^\infty \frac{z^n_{s-}}{z^n_{s-}}\id_{\{z^n_{s-}>0\}}K^n_s d\cro{z^n,N}_s\right]\\
&=&\E\left [\int_0^\tau \frac{1}{Z_{s-}}J_s d \cro{m,N}_s\right ]
+\sum_{n=1}^\infty\E\left [\id_{C_n}\int_{0}^\infty \frac{1}{z^n_{s-}}K^n_s d\cro{z^n,N}_s\right].
\end{eqnarray*}
The assertion of the theorem follows as, for any $s\leq t$ and $F\in \F^\tau_s$, the process $H=\id_{(s,t]}\id_{F}$ is clearly $\ff^\tau$-predictable.
To end the proof, we recall that any local martingale is locally in $H^1$ (see \cite[Theorem 51, Chapter IV]{protter}).
\end{proof}

\begin{remark}
Lemma (4,11) in \cite{j}, where the random time with countably many
values is considered, is  a special case of Theorem
\ref{decomposition1}. It corresponds to the situation of thin
random time whose graph is included in countable union of constant
sections, i.e, $\graph{\tau}\subset \bigcup_{n}\graph{t_n}$ with
$\graph{t_n}=\{(\omega, t_n): \omega\in \Omega\}$ and $t_n\in\mathbb R$.
\end{remark}

We end this section with a second proof of Theorem \ref{decomposition1} which is based on  result linking processes in $\ff^\tau$
and $\ff^{ \, \mathcal C}$ stated in Lemma \ref{project}.

\emph{Second proof of Theorem \ref{decomposition1}}.
Let $X$ be a bounded $\ff$-martingale. Then, it is enough to show that $\id_{\Rbrack 0, \tau \Lbrack}\bigcdot X$ and $\id_{\Lbrack \tau, \infty \Rbrack}\bigcdot X$ are two $\ff^\tau$-semimartingales with appropriate decompositions.
By \cite[Proposition (4,16)]{j}
$\id_{\Rbrack 0, \tau \Lbrack}\bigcdot X=X_{\cdot\land \tau}$ is an $\ff^\tau$-semimartingale with the decomposition
$$X_{t\land \tau}=\wh X^1_t+\int_0^{t\land
\tau}\frac{1}{Z_{s-}}d\cro{X,m}_s,$$
where $\wh X^1$ is an $\ff^\tau$-local martingale.
By Lemma \ref{project} and Jacod's result \eqref{jacod_result} it follows that $\id_{\Lbrack \tau, \infty \Rbrack}\bigcdot X$ is an $\ff^\tau$-semimartingale with  decomposition
$$X_t-X_{t\land \tau}=\wh X^2_t
+\sum_{n=1}^\infty \id_{C_n} \int_0^t
\id_{\{s>T_n\}}\frac{1}{z^n_{s-}}d\cro{X,z^n}_s,
$$
where $\widehat X^2$ is an $\ff^\tau$-local martingale. This completes the proof.
\qed

\begin{lemma}
\label{project}
Let $\tau$ be a thin time and $Y$ be a process such that $Y=\id_{\Lbrack \tau, \infty \Rbrack}\bigcdot Y$. Then:\\
(a) The process $Y$ is an $\ff^{ \, \mathcal C}$-(super, sub)martingale if and only if the process $Y$ is an $\ff^\tau$-(super, sub)martingale.\\
(b) Let $\vartheta$ be an $\ff^{ \, \mathcal C}$-stopping time. Then $\vartheta\lor \tau$ is an $\ff^\tau$-stopping time.\\
(c) The process $Y$ is an $\ff^{ \, \mathcal C}$-local martingale if and only if the process $Y$ is an $\ff^\tau$-local martingale.
\end{lemma}

\begin{proof}
(a) Note that the filtrations $\ff^\tau$ and $\ff^{ \, \mathcal C}$ are equal after $\tau$, i.e.,
for each $t$ and for each set $G\in\F^\mathcal C_t$, there exists a set $F\in \F^\tau_t$ such that
\begin{equation}
\label{2fil}
\{\tau\leq t\}\cap G=\{\tau\leq t\}\cap F .
\end{equation}
To show \eqref{2fil}, by monotone class theorem, it is enough to consider $G=C_n$ and to take $F=C_n\cap\{\tau\leq t\}$ which belongs to $\F^\tau_t$ as $C_n\in \F^\tau_\tau$ by \cite[Corollary 3.5]{chinois}.
That implies that the process $\id_{\Lbrack \tau, \infty \Rbrack}\bigcdot Y$ is $\ff^\tau$-adapted if and only if it is $\ff^{ \, \mathcal C}$-adapted.
The equivalence of (super-, sub-) martingale property comes from \eqref{2fil}.\\
(b) For each $t$ we have $\{\vartheta\lor\tau\leq t\}=\{\vartheta\leq t\}\cap\{\tau\leq t\}\in \F^\tau_t$ by \eqref{2fil}.\\
(c) We combine the two previous points.
\end{proof}

\section{Immersion for thin times}
\label{s:thinimm}

Immersion, also called the hypothesis $(\mathcal H)$ is a more restrictive hypothesis for enlargement of filtration
 then the hypothesis $({\mathcal H}^\prime)$.
Given $\ff\subset \gg$, we say that $\ff$ is immersed in $\gg$ if any $\ff$-martingale is a $\gg$-martingale.
The equivalent condition to immersion, established in Theorem 3 in \cite{BY}, says that for each $t\geq 0$ and $G\in L^1(\G_t)$ it holds that $\E[G|\F_t]=\E[G|\F_\infty]$.
Immersion does not hold for each thin time. However, in the next proposition, an equivalent condition to immersion is given.
In particular it implies that there exist thin times for which immersion holds and which are not stopping times.

\begin{proposition}
\label{H:thin}
Let $\tau$ be a thin time with exhausting sequence $(T_n)_{n\geq 0}$, partition $(C_n)_{n\geq 0}$ and martingale family $(z^n)_{n\geq 0}$. Then, $\ff$ is immersed in $\ff^\tau$ if and only if one of the following conditions hold:
\begin{enumerate}
\item[(a)] $z^n_\infty=z^n_{T_n}$ for each $n\geq 1$,
\item[(b)] $z^n_t=z^n_{T_n\land t}$ for each $t\geq 0$ for each $n\geq 1$,
\item[(c)] for each $n\geq 1$, $C_n$ is independent of $\F_\infty$ conditionally w.r.t. $\F_{T_n}$.
\end{enumerate}
\end{proposition}
\begin{proof}
By Theorem 3 in \cite{BY}, Lemma \ref{key} (a) and monotone class theorem, $\ff$ is immersed in $\ff^\tau$ if and only if $\P(C_n\cap \{T_n\leq t\}|\F_t)=\P(C_n\cap \{T_n\leq t\}|\F_\infty)$ for each $n\geq 1$.
The last condition is precisely $z^n_t\id_{\{T_n\leq t\}}=z^n_\infty\id_{\{T_n\leq t\}}$ for each $n\geq 1$, which is the condition (b).
Since $z^n$ are martingales, we conclude that immersion is equivalent to $z^n_{T_n}=z^n_\infty$ stated in the condition (a).
Since $z^n_{T_n}=z^n_\infty$ can be rewritten as $\P(C_n|\F_{T_n})=\P(C_n|\F_\infty)$, we conclude that immersion is satisfied if and only if, for each $n\geq1$, $C_n$ is independent of $\F_\infty$ conditionally w.r.t. $\F_{T_n}$.
\end{proof}

\begin{remark}
Immersion property for a random time $\tau$ implies in particular that $\tau$ is a pseudo-stopping time which where studied in \cite{williams02} and \cite{NY1}.
A random time $\tau$ is a pseudo-stopping time if for any bounded $\ff$-martingale $X$ it holds that $\E[X_\tau]=\E[X_0]$, or equivalently as established in \cite{NY1}, if $m\equiv 1$.

Let $\tau$ be a thin time. Then, by Proposition \ref{positive}, $\tau$ is a pseudo-stopping time if and only if $\sum_{n=0}^\infty z^n_{t\land T_n}=1$ for any $t\geq 0$. Clearly, immersion implies the last condition as, by Proposition \ref{H:thin} (b) and since $T_0=\infty$,
$$\sum_{n=0}^\infty z^n_{t\land T_n}=z^0_t+\sum_{n=1}^\infty z^n_{t\land T_n}=\sum_{n=0}^\infty z^n_{t}=1.$$
Reverse implication does not hold.
\end{remark}

\section{Thin-thick decomposition of a random time}
\label{s:decomp}

In this section we present an application of thin times to the decomposition of a generic random time into thin and thick parts. In the first subsection we introduce and present some result about thick times.
Then, in the second subsection, we establish the thin-thick decomposition. Finally, in the remaining subsections, we apply thin-thick decomposition to obtain results on the hypothesis $(\mathcal H^\prime)$ and immersion.

\subsection{Thick times}
As described in the introduction, thick times avoid stopping times from the reference filtration, i.e., thick times are defined in the following way.
\begin{definition}
\label{thick_defi}
A random time $\tau$ is called
a thick time if $\graph {\tau}  \cap\graph  {T}$ is evanescent for any $\ff$-stopping time $T$, i.e., if it  avoids all $\ff$-stopping times.
\end{definition}
Similarly as for thin times in Theorem \ref{A0jump}, thick times can be characterized in terms of their dual optional projection.
\begin{theorem}
\label{A0cont}
A random time is a thick time if and only if its dual optional projection is a continuous process.
In that case $A^o=A^p$.
\end{theorem}

\begin{proof}
Let $T$ be an $\ff$-stopping time.
Since $\E[\Delta A^o_T\id_{\{T<\infty\}}]=\P(\tau=T<\infty)$ and $A^o$ is an
increasing process, we deduce that
$$\P(\tau=T<\infty)=0 \quad \textrm{if and only
if}  \quad \Delta A^o_T\id_{\{T<\infty\}}=0\;\;\; \P\textrm{-a.s.}
$$
Since $\{\Delta A^o>0\}$ is an optional set, the optional section theorem \cite[Theorem 4.7]{chinois} implies that
$\{\Delta A^o>0\}$ is exhausted by disjoint graphs of $\ff$-stopping times.
Thus, we conclude that $\tau$ is a thick time if and only if $A^o$ is continuous.
\end{proof}

The straightforward observation that the two  classes of thin and thick
times have trivial intersection is stated in the following lemma.

\begin{lemma}
A random time $\tau$  belongs to the class of thick times and to the class of thin times if and only if $\tau=\infty$.
\end{lemma}

\subsection{Decomposition of a random time}

The main concept of this section, the thin-thick decomposition, is presented in the next definition.
 It is followed by the result stating the existence of such a decomposition for any random time.

\begin{definition}
\label{basic}
Consider a random time $\tau$.
A pair of random times $(\tau_1,\tau_2)$ is called a thin-thick decomposition of $\tau$ if
$\tau_1$ is a thin time, $\tau_2$ is a thick time, and
$$\tau=\tau_1\land \tau_2 \quad \quad \tau_1\lor\tau_2=\infty.$$
\end{definition}

\begin{theorem}
\label{A0decom}
Any random time $\tau$ has a thin-thick decomposition $(\tau_1, \tau_2)$ which is unique on the set $\{\tau<\infty\}$.
\end{theorem}
\begin{proof}
Let us define  $\tau_1$ and $\tau_2$  as
$\tau_1:=\tau_{\{\Delta A^o_\tau >0 \}}$ and
$\tau_2:=\tau_{\{\Delta A^o_\tau =0 \}},$ where $\tau_C$ is the
restriction of the random time $\tau$ to the set $C$, defined as
$\tau_C=\tau\id_{C}+\infty\id_{C^c}.$ Properties of dual optional
projection ensure that $\tau_1$ and $\tau_2$ satisfy the required
conditions. More precisely, the time $\tau_1$ is a thin time since
\begin{align*}
\graph{\tau_1}=\graph{\tau}\cap\{\Delta A^o>0\}
= \graph{\tau}\cap\bigcup_n \ \graph{T_n} \subset \ \bigcup_n \ \graph{T_n},
\end{align*}
where the sequence $(T_n)_n$ exhausts the jumps of the c\`adl\`ag
increasing process $A^o$, i.e., $\{\Delta
A^o>0\}=\bigcup_n \graph{T_n}$ and the time $\tau_2$ is a thick time
since, for any $\ff$-stopping time $T$,
\begin{align*}
\P(\tau_2=T<\infty)&
=\E\left[\id_{\{\tau=T\}\cap\{\Delta A^o_\tau =0 \}}\id_{(T<\infty)}\right]\\
&=\E\left[\int_0^\infty \id_{\{u=T\}\cap\{\Delta A^o_u =0 \}}dA^o_u\right]=0.
\end{align*}.
\end{proof}

For $i\in\{1, 2\}$, corresponding to the thin part $\tau_1$ and the thick part $\tau_2$ of a random time $\tau$,
we define $A^i:=\id_{\Rbrack \tau_i, \infty\Rbrack}$. Then $A^{i, p}$ and $A^{i, o}$ are respectively the $\ff$-dual predictable projection and the $\ff$-dual optional projection of $A^i$.
Let us denote by $Z^i$ and $\wt Z^i$ the supermartingales associated with $\tau_i$.
Then, the following relations hold.

\begin{proposition}
\label{supermartingales}
Let $\tau$ be a random time and $(\tau_1,\tau_2)$ its thin-thick decomposition.
(a) The supermartingales $Z$ and $\widetilde Z$ can be decomposed in terms of
the supermartingales $Z^{1}$, $Z^{2}$ and $\widetilde Z^{1}$, $\widetilde Z^{2}$ as:
\[
Z=Z^{1}+Z^{2}-1 \quad \textrm{and} \quad
\widetilde Z=\widetilde Z^{1}+\widetilde Z^{2}-1.
\]
(b) The dual optional projection $A^o$ can be decomposed as $A^o=A^{1,o}+A^{2,o}$.\\
(c) The dual predictable projection $A^p$ can be decomposed as $A^p=A^{1,p}+A^{2,p}$.
\end{proposition}

\begin{proof}
The result follows from $\id_{\Rbrack \tau_1,\infty \Rbrack}+\id_{\Rbrack \tau_2,\infty \Rbrack}=\id_{\Rbrack \tau,\infty \Rbrack}$ which holds since $\tau_1\lor\tau_2=\infty$.
\end{proof}

\begin{lemma}
\label{l:59}
Let $\tau$ be a random time and $(\tau_1,\tau_2)$ its thin-thick decomposition.
Let $(T_n)_{n\geq 0}$ be an $\ff$-exhausting sequence, $(C_n)_{n\geq 0}$ an $\ff$-partition and $(z^n)_{n\geq 0}$ an $\ff$-martingale family of the $\ff$-thin time $\tau_1$.
Then, for any $t\geq 0$,
\begin{align*}
\P( C_n\vert \F^{\tau_2}_t )&=\id_{\{t<\tau_2 \}}
 \frac {z^n_t }{Z^2_t}\quad \textrm{for all} \quad n\geq 0\,,\\
\P(\tau_1>t|\F^{\tau_2}_t)&=1-\id_{\{t<\tau_2\}}
\frac{1-Z^1_t}{Z^2_t},\\
\P(\tau_2>t|\F^{\tau_1}_t)&=1-\id_{\{t<\tau_1\}}
\frac{1-Z^2_t}{Z^1_t}\,.
\end{align*}
\end{lemma}
\begin{proof}
Let us compute $\P(C_n|\F^{\tau_2}_t)$ on the two sets before
$\tau_2$ and after $\tau_2$ separately as follows:
$$\P(C_n|\F^{\tau_2}_t)= \id_{\{\tau_2\leq t \}}\P( C_n\vert \F _t  \vee \sigma (\tau_2))+
\id_{\{t<\tau_2 \}} \frac{\P( C_n\cap\{t<\tau_2\}\vert \F _t )}{\P( t<\tau_2\vert \F _t) }\,.$$
Then, taking into account that $C_n=\{\tau_1=T_n<\infty\}$ and $\tau_1 \vee
 \tau_2=\infty$, one has on the one hand
 $$\id_{\{\tau_2\leq t\}} \P(C_n
\vert \F_t \vee \sigma (\tau_2))= \P(\tau_1=T_n<\infty, \tau_2\leq t\vert
\F_t \vee \sigma (\tau_2))= 0\,,$$
and, on the other hand,
$$\P( C_n\cap\{t<\tau_2\}\vert \F _t )=\P( C_n\vert \F _t )=z^n_t.$$
Therefore, the first equality holds.

By symmetry we only prove the second identity and we skip the third one.
Similarly as before we compute $\P(\tau_1>t|\F^{\tau_2}_t)$ on the two sets before
$\tau_2$ and after $\tau_2$ separately as follows:
$$\P(\tau_1>t|\F^{\tau_2}_t)= \id_{\{\tau_2\leq t\}} \P(\tau_1>t
\vert \F_t \vee \sigma (\tau_2))+\id _{\{t<\tau_2\}}
 \frac{\P(\tau_2>t,\tau_1>t\vert \F_t)}{\P(\tau_2>t \vert \F_t)}\,.$$
Then, on the one hand, taking into account that $\tau_1 \vee
 \tau_2=\infty$,
 $$\id_{\{\tau_2\leq t\}} \P(\tau_1>t
\vert \F_t \vee \sigma (\tau_2))= \P(\tau_1>t \geq \tau_2 \vert
\F_t \vee \sigma (\tau_2))= \P(  t\geq \tau_2 \vert \F_t \vee
\sigma (\tau_2))= \id_{\{\tau_2\leq t\}}\,,$$
and, on the other hand, by $\tau_1\land\tau_2=\tau$ and Proposition \ref{supermartingales} (a),
$$\P(\tau_1>t,\tau_2>t\vert \F_t)=\P(\tau>t \vert
\F_t)=Z_t=Z^1_t+Z^2_t-1.$$
Therefore,
$$\P(\tau_1>t|\F^{\tau_2}_t)=\id_{\{\tau_1\leq t\}}+\id_
{\{t<\tau_2\}}\frac{Z^1_t+Z^2_t-1}{Z^2_t}$$ and the result
follows.
\end{proof}

The following proposition we study the condition $A^o=A^p$. We have seen already that, if either $\tau$ avoids $\ff$ stopping times or all $\ff$-martingales are continuous, then this condition holds. 
\begin{proposition} 
\label{pro:oa}
The condition $A^o=A^p$ holds if and only if the random time $\tau$ satisfies that:
\begin{enumerate}
\item 
$\P(\tau=T<\infty)=0$ for any $\ff$-totally inaccessible stopping time $T$,
\item $\P(\tau=S<\infty|\F_S)=\P(\tau=S<\infty|\F_{S-})$ for any $\ff$-predictable stopping time $S$.
\end{enumerate}
\end{proposition}
\begin{proof}
By Proposition \ref{supermartingales} and Theorem \ref{A0cont} it is enough to assume that $\tau$ is a thin time. 
We choose the exhausting sequence $(S_n)_{n\geq 1}$ so that it only contains totally inaccessible or predictable stopping times. Then, by Proposition \ref{positive} (c) and by the fact that $A^p=(A^o)^p$ we conclude that $A^o=A^p$ is equivalent to $z^n_{S_n}\id_{\Rbrack S_n, \infty\Rbrack}=\left(z^n_{S_n}\id_{\Rbrack S_n, \infty\Rbrack}\right)^p$ for each $n\geq 1$. 
If $S_n$ is totally inaccessible then the latter condition is equivalent to $z^n=0$ which is the condition (1). 
If $S_n$ is predictable then $\left(z^n_{S_n}\id_{\Rbrack S_n, \infty\Rbrack}\right)^p=\, ^pz^n_{S_n}\id_{\Rbrack S_n, \infty\Rbrack}$ which boils down to the condition (2).
\end{proof}
\begin{remark}
The condition (2) from Proposition \ref{pro:oa} is always satisfied in quasi-left continuous filtrations since then $\F_S=\F_{S-}$ for $\ff$-predictable stopping time $S$.
Therefore, if $\ff$ is a natural filtration of a Poisson process with jump times $(T_n)_{n\geq 1}$, which is quasi-left continuous, the condition $A^o=A^p$ holds if and only if $\P(\tau=T_n<\infty)=0$ for any $n$ since any $\ff$-totally inaccessible stopping time $T$ satisfies $[\![T]\!]\subset \bigcup_n [\![T_n]\!]$.
One can  find such an example in \cite[Proposition 4]{ACDJ_pekin}.
\end{remark}

\begin{remark}
(a) Since $\tau_1$ is an $\ff^\tau$-stopping time, it can be decomposed into $\ff^\tau$-accessible and $\ff^\tau$-totally inaccessible parts.
Thus, we can consider the decomposition of $\tau$ into three parts as:
\[
\tau^i_1=\tau_{\{\Delta A^o_\tau >0 , \  \Delta A^p_\tau=0\}}, \quad
\tau^a_1=\tau_{\{\Delta A^o_\tau >0, \ \Delta A^p_\tau>0\}} \quad \textrm{and} \quad \tau_2=\tau_{\{\Delta A^o_\tau =0\}}.
\]
Since $\tau_2$ is $\ff^\tau$-totally inaccessible, it follows that
$\tau^i_1\land\tau_2$ is the $\ff^\tau$-totally inaccessible part
and $\tau^a_1$ is the $\ff^\tau$-accessible part of the
$\ff^\tau$-stopping time $\tau$. Results of a similar type can be found in \cite{delia}  and \cite[p.65]{j}. We note that $\tau$ is
an $\ff^\tau$-predictable stopping time if and only if $\tau$ is
an $\ff$-predictable stopping time.\\
(b) Assume that the $\ff^\tau$-accessible stopping time $\tau^a_1$
is not an $\ff$-stopping time. Then, the filtration $\ff^\tau$ is not quasi-left continuous. This provides a systemic way to construct examples of
non quasi-left continuous filtrations.
\end{remark}

\subsection{The hypothesis $({\mathcal H}^\prime)$ for a random time}
\label{generalH}

We study here the hypothesis $({\mathcal H}^\prime)$ in the
progressive enlargement of filtration in  connection to the
thin-thick decomposition of the random time.
Let  $(\tau_1,
\tau_2)$ be the thin-thick de\-com\-po\-si\-tion of  a random time $\tau$. We define  three enlarged
filtrations $\ff^{\tau_1}:=(\F^{\tau_1}_t)_{t\geq 0}$,
$\ff^{\tau_2}:=(\F^{\tau_2}_t)_{t\geq 0}$ and $\ff^{\tau_1,
\tau_2}:=(\F^{\tau_1, \tau_2}_t)_{t\geq 0}$ as
\begin{align*}
\F^{\tau_i}_t:&=\bigcap_{s>t}\F_s\lor \sigma(\tau_i\land s)\quad\textrm{for} \quad i=1,2 \\
\F^{\tau_1, \tau_2}_t:&=\bigcap_{s>t}\F_s\lor \sigma(\tau_1\land s)\lor \sigma(\tau_2\land s).
\end{align*}
Clearly, $\ff\subset\ff^{\tau_i}\subset \ff^{\tau_1,
\tau_2}=(\ff^{\tau_1})^{ \tau_2}=(\ff^{\tau_2})^{ \tau_1} $ for
$i=1,2$.
\begin{theorem}
\label{Hp:thth}
Let $\tau$ be a random time and $(\tau_1,\tau_2)$ its thin-thick decomposition.
Then, $\ff^\tau = \ff^{\tau_1,\tau_2}$. Furthermore,  the hypothesis $({\mathcal H}^\prime)$ is satisfied for $(\ff, \ff^\tau)$ if and only if the hypothesis $({\mathcal H}^\prime)$ is satisfied for $(\ff, \ff^{\tau_2})$.
\end{theorem}
\begin{proof}
In a first step, we show that, for $i=1,2$:
\[
\ff\subset \ff^{\tau_i}\subset \ff^{\tau_1, \tau_2}=\ff^\tau.
\]
Let $A^o$ be the $\ff$-dual optional projection of $\tau$. Note
that
\[
\id_{\Rbrack \tau_1, \infty\Rbrack}=\id_{\Rbrack \tau, \infty\Rbrack}\id_{\{\Delta A^o_\tau>0\}} \quad \textrm{and} \quad
\id_{\Rbrack \tau_2, \infty\Rbrack}=\id_{\Rbrack \tau, \infty\Rbrack}\id_{\{\Delta A^o_\tau=0\}},
\]
thus, since $\Delta A^o_\tau\in \F^\tau_\tau$, the processes
$\id_{\Rbrack \tau_1, \infty\Rbrack}$ and $\id_{\Rbrack \tau_2,
\infty\Rbrack}$ are $\ff^\tau$-adapted which implies that
$\ff^{\tau_1, \tau_2}\subset \ff^\tau$. On the other hand, we have
\[
\id_{\Rbrack \tau_1, \infty\Rbrack}+\id_{\Rbrack \tau_2, \infty\Rbrack}=\id_{\Rbrack \tau, \infty\Rbrack}
\]
which implies that $\ff^{\tau_1, \tau_2}\supset \ff^\tau$.

In a second step, note that if an $\ff$-martingale is an
$\ff^\tau$-semimartingale, by Stricker's Theorem \cite[Theorem 4,
Chapter II, p. 53]{protter}, it is as well an
$\ff^{\tau_2}$-semimartingale. Thus the necessary condition
follows. Since $\tau_1$ is an $\ff$-thin time, it is an $\ff^{\tau_2}$-thin time and the equality
$\ff^{\tau_1, \tau_2}=\ff^\tau$ and Theorem \ref{decomposition1}
imply that the hypothesis $({\mathcal H}^\prime)$ is satisfied for
$(\ff^{\tau_2}, \ff^\tau)$. Thus the sufficient condition follows.
\end{proof}

In the following corollary we examine the hypothesis $({\mathcal H}^\prime)$ a minimum of a thin time and random times satisfying the hypothesis $({\mathcal H}^\prime)$, namely honest times and times satisfying Jacod's absolute continuity condition (see \cite[Chapter 5]{j} and  \cite{jeanblanc2009progressive, jacod1985grossissement} respectively).
\begin{corollary}
Let $\tau$ be a thin time and $\sigma$ be an honest time or satisfies Jacod's absolute continuity condition. Then, the hypothesis $({\mathcal H}^\prime)$ is satisfied for $(\ff, \ff^{\tau\land\sigma})$
\end{corollary}

\begin{proof}
First we recall that, if $\sigma$ is honest then the hypothesis $({\mathcal H}^\prime)$ is satisfied for $(\ff, \ff^{\sigma})$ by \cite[Theorem (5,10)]{j}, and if $\sigma$ satisfies Jacod's absolute continuity condition then $({\mathcal H}^\prime)$ is satisfied for $(\ff, \ff^{\sigma})$ by \cite[Theorem 3.1]{jeanblanc2009progressive}.

Let $(\sigma_1, \sigma_2)$ be a thin-thick decomposition of $\sigma$. Then, by Remark \ref{r:stability}, $\tau\land\sigma_1$ is a thin time and $(\tau\land\sigma_1, \sigma_2)$ is a thin-thick decomposition of $\tau\land \sigma.$
Then the statement of the corollary follows by applying twice Theorem \ref{Hp:thth}.
\end{proof}

\begin{proposition}
Let $\tau$ be a random time and $(\tau_1,\tau_2)$ its thin-thick decomposition.
Let $(T_n)_{n\geq 0}$ be an $\ff$-exhausting sequence, $(C_n)_{n\geq 0}$ an $\ff$-partition and $(z^n)_{n\geq 0}$ an $\ff$-martingale family of $\ff$-thin time $\tau_1$.
Assume that for an $\ff$-martingale $X$, there
exists an $\ff$-predictable finite variation process  $\Gamma (X)$ such that $X= \widetilde X+
\Gamma (X)$ where $\widetilde X$ is an
$\ff^{\tau_2}$-martingale. Then,
\begin{eqnarray*} X_t&=& \widehat X_t+ \Gamma (X)_t+ \int_0^{t\land
\tau}\frac{1}{Z_{s-}}d{\cro{ \widetilde X,
\widetilde m}}^{\ff^{\tau_2}}_s+\sum_{n=1}^\infty \id_{C_n} \int_0^t
\id_{\{s>T_n\}}\frac{1}{\widetilde z^{n}_{s-}}d{\cro{\widetilde X,\widetilde z^{n}}}^{\ff^{\tau_2}}_s, \end{eqnarray*}
 where $\widehat X$ is an $\ff^\tau$-martingale, $ \widetilde z^{n}_t:=\P( C_n\vert \F^{\tau_2}_t)=\id_{\{t<\tau_2 \}}
 \frac {z^n_t }{Z^2_t}$ and $\widetilde m_t=\sum_n \widetilde z^{n}_{t\wedge T_n}\,.$
 \end{proposition}

\begin{proof}
The decomposition of $X$ as an $\ff^\tau$-semimartingale follows by  $\ff^\tau = \ff^{\tau_1,\tau_2}$ and Theorem \ref{decomposition1} since $\tau_1$ is an $\ff^{\tau_2}$-thin time. 
Lemma \ref{l:59} and Proposition \ref{positive} imply the forms of $\widetilde z^n$ and $\widetilde m$.
\end{proof}

\subsection{Immersion for a random time}
Thin-thick decomposition finds application in studying immersion for a generic random time.
\begin{proposition} $\ff$ is immersed in $\ff^\tau$ if and only if
$\ff$ is immersed in $\ff^{\tau_1}$ and in $\ff^{\tau_2}$. In that
case, $\ff^{\tau_1}$ and  $\ff^{\tau_2}$
 are immersed in $\ff^\tau$.\end{proposition}
 \begin{proof}
Since $\ff\subset \ff^{\tau_i}\subset \ff^\tau$, it is clear that, if $\ff$
is immersed in $\ff^\tau$, then $\ff$ is immersed in
$\ff^{\tau_1}$ and in $\ff^{\tau_2}$.

Let $\ff$ be immersed in $\ff^{\tau_1}$ and in $\ff^{\tau_2}$, i.e., $Z^i_t=\P(\tau_i>t|\F_\infty)$ for each $t\geq 0$, for $i=1,2$. Then, by Proposition \ref{supermartingales} (a),
$$ Z_t =Z^1_t+Z^2_t-1= \P(\tau_1>t \vert \F_\infty)+\P(\tau_2>t \vert
\F_\infty)-1=\P(\tau>t\vert \F_\infty)$$
and we conclude that $\ff$ is immersed $\ff^\tau$.

It remains to prove the last assertion. Let $\ff$ be immersed in $\ff^{\tau}$.
Then, using similar arguments as in the proof of Lemma \ref{l:59}, we obtain:
$$\P(\tau_2>t \vert \F^{\tau_1}_\infty)= \id_{\{ \tau_1 \leq t\}}+ \id_{\{t<\tau_1\}}
\frac{\P( \tau>t\vert \F_\infty)}{\P( \tau_1>t\vert \F_\infty)}
$$ and the assumed immersion yield to
$$\P(\tau_2>t|\F^{\tau_1}_\infty)=\id_{\{ \tau_1 \leq t\}}+ \id_{\{t<\tau_1\}}
\frac{\P( \tau>t\vert \F_t)}{\P( \tau_1>t\vert \F_t)}=
\P(\tau_2>t \vert \F_t^{\tau_1})\,.
$$
Therefore, $\ff^{\tau_1}$ is immersed in $\ff^\tau$. The same proof is valid for $\tau_2$.
\end{proof}

\begin{com}   
In  \cite{jl:ms},  the authors
introduce  a random time $\tau =\vartheta \wedge \xi$ where $\xi$
avoids $\ff$-stopping times, and is constructed as
$\xi=\inf\{t\,:\, \Lambda _t:=\int_0^t \lambda_sds \geq \Theta\}$
where $\lambda$ is a positive $\ff$-adapted process  and $\Theta$
is an exponential random variable independent from $\ff$, and
$\vartheta$ is an $\ff$-accessible stopping time. Therefore,   $\xi$ is thick  and
$\vartheta$ is thin. The thin-thick decomposition   $\tau= \tau_1 \wedge\tau_2$   can be obtained as follows:  $\tau_2= \vartheta \id_{\{ \vartheta <\xi\}}+\infty\id_{\{
\xi \leq \vartheta \}}$ and $\tau_1= \xi \id_{\{\xi
<\vartheta\}}+\infty\id_{\{\vartheta \leq \xi \}}$.  \\
  The authors   establish  immersion property   by checking $\P(\tau>t \vert \F_t)=\P(\tau>t \vert \F_\infty)$, a   characterisation of immersion that we have recalled above. From our result,
immersion holds  since $\ff$ is immersed in $\ff^\xi$, hence in   $\ff^1$, due to the property that $\vartheta$ is an $\ff$-stopping time.
\end{com}\section{Link between thin times and honest times}
 \label{s:honest}

 In this
section we restrict our attention to a special class of random
times, namely to honest times. We recall the definition below (see
\cite[p. 73]{j}) and some alternative characterizations in
Appendix \ref{honest_appendix}.
Honest times are well- studied class of time fo which, in particular, the hypothesis $({\mathcal H}^\prime)$ holds.

\begin{definition}
\label{defhonest} A random time $\tau$ is an $\ff$-honest time if
for every $t > 0$ there exists an $\F_t$-measurable random
variable $\tau_t$ such that $\tau=\tau_t$ on $\{\tau<t\}$.
Then, it is always possible to choose $\tau_t$ such that
$\tau_t\leq t$.
\end{definition}


\subsection{Fundamental properties}

Let us start with some characterisation and properties of (thin) honest times.

\begin{theorem}
\label{condZ1} (a) Let $(\tau_1 ,\tau_2)$ be the
thin-thick decomposition of $\tau$. Then, $\tau$ is honest if and only
if $\tau_1$ and $\tau_2$ are honest.\\ (b) A random time $\tau$
is a thick  honest time if and only if $Z _\tau=1$
a.s. on $\{\tau<\infty\}$. \\
(c) Let $\tau$ be an honest time with thin-thick decomposition $(\tau_1,\tau_2)$. Then,
$Z_\tau<1$ on $\{\tau=\tau_1<\infty\}$ and
$Z_\tau=1$ on $\{\tau=\tau_2<\infty\}$.
\end{theorem}

\begin{proof}
(a) On the set $\{\tau<\infty\}$, $\tau$ is equal to $\gamma$, the
end of the optional set $\Gamma$ (Theorem \ref{jeulin5,1}). Then,
as $\{\tau_1 <\infty\}\subset\{\tau<\infty\}$, on the set
$\{\tau_1 <\infty\}$, one has $\tau_1=\gamma$, so $\tau_1$ is an
honest time. Same argument for $\tau_2$.\\
(b) Assume that $\tau$ is a thick  honest time. Then,
the honest time property presented in Theorem \ref{jeulin5,1} (c)
implies that $\widetilde Z_\tau=1$ and the thick  time property
implies, by Theorem \ref{A0jump} (b), the continuity of $A^o$.
Therefore, the equality $\widetilde Z= Z+ \Delta A^{o }$ leads to equality $Z_\tau=1$ a.s. on $\{\tau<\infty\}$.

Assume now  that $Z _\tau=1$ on the set $\{\tau<\infty\}$. Then, on $\{\tau<\infty\}$ we have
$1=Z _\tau\leq \widetilde Z_\tau \leq 1$, so $\widetilde Z_\tau=1$
and $\tau$ is an honest time. Furthermore, as $\Delta
A^o_\tau=\widetilde Z_\tau-Z_\tau=0$, for each ${\mathbb F}$-stopping time $T$ we have
\begin{eqnarray*}
{\mathbb P}(\tau=T<\infty)&=&\E\left[{1\!\!1}_{\{\tau=T\}}{1\!\!1}_{\{\Delta A^o_\tau =0 \}}{1\!\!1}_{(T<\infty)}\right]\\
&=&\E\left [\int_0^\infty {1\!\!1}_{\{u=T\}}{1\!\!1}_{\{\Delta A^o_u =0 \}}dA^o_u\right]=0.
\end{eqnarray*}
So $\tau$ is a thick time.\\
(c) From the honest time property of $\tau$ and Proposition \ref{supermartingales} (a), on the set $\{\tau<\infty\}$
\[
1=\widetilde Z_\tau=\widetilde Z^{1}_{\tau}+\widetilde Z^{2}_{\tau}-1.
\]
On the set $\{\tau=\tau_1<\infty\}$,
\begin{align*}
Z_\tau&=Z^{1}_{\tau_1}+Z^{2}_{\tau_1}-1\leq Z^{2}_{\tau_1}<1,
\end{align*}
where  the last inequality is due to Proposition \ref{positive} (b).
On the set $\{\tau=\tau_2<\infty\}$, we have
\begin{align*}
1&=\widetilde Z^{1}_{\tau_2}+\widetilde Z^{2}_{\tau_2}-1=\widetilde Z^{1}_{\tau_2},
\end{align*}
where the second equality comes from  (c) in  Theorem \ref{jeulin5,1}.
Now let us compute $Z^{1}_{\tau_2}$
\[
Z^{1}_{\tau_2}=\widetilde Z^{1}_{\tau_2}-\Delta A^{1\,,o}_{\tau_2}=\widetilde Z^{1}_{\tau_2}=1,
\]
where we have used the thick time property of $\tau_2$,
i.e., $\{\Delta A^{1\,,o}>0\}=\bigcup_{n=1}^\infty\graph{T_n}$ (with $(T_n)_{n\geq 0}$
being an exhausting sequence of $\tau_1$) and
$\P(\tau_2=T_n<\infty)=0$.
Finally, on $\{\tau=\tau_2<\infty\}$
\[
Z_\tau=Z^{1}_{\tau_2}+Z^{2}_{\tau_2}-1=1.
\]
\end{proof}

\begin{remark}
We would like to remark that the condition that $Z_\tau<1$ for an honest time $\tau$ -- which, by Theorem \ref{condZ1} (c),  is equivalent to the condition that $\tau$ is a thin honest time -- is an essential assumption in \cite{ACDJ_after} for the study of arbitrages after honest times.
\end{remark}

\begin{lemma}
\label{thinhonest}
Let $\tau$ be a thin honest time, $\tau_t$ be associated with $\tau$ as in Definition \ref{defhonest} and $(T_n)_{n\geq 0}$ be an exhausting sequence of $\tau$.
Then:\\
(a) on $\{T_n=\tau_t\}=\{T_n=\tau_t\leq t\}$ we have
 $z^n_t=1-Z_t$, $A^o_t=z^n_{T_n}$ and $1-m_t=z^n_t-z^n_{T_n}$ for each $n\geq 1$;\\
(b) on $\{T_n<t\}$ we have $z^n_t=\id_{\{\tau_t=T_n\}}(1-\wt Z_t)$ and $z^n_{t-}=\id_{\{\tau_t=T_n\}}(1-Z_{t-})$ for each $n\geq 1$;
in particular
$$1-\wt Z_t=\sum_{n=1}^\infty \id_{\{\tau_t=T_n<t\}}(1-\wt Z_t)\quad  \textrm{and} \quad 1-Z_{t-}=\sum_{n=1}^\infty\id_{\{\tau_t=T_n<t\}}(1-Z_{t-}).$$
\end{lemma}

\begin{proof}
(a) Using properties of $\tau_t$ we deduce that
\begin{align*}
\id_{\{T_n =\tau_t\}}z^n_t&=\P(T_n=\tau_t\leq t, \tau=T_n<\infty|\F_t)\\
&=\P(\tau\leq t, T_n =\tau_t=\tau|\F_t)\\
&=\P(\tau\leq t, T_n=\tau_t|\F_t)\\
&=\id_{\{T_n=\tau_t\}}(1-Z_t)
\end{align*}
where the first equality is due to $\tau_t\leq t$, the third one follows by $\tau_t=\tau$ on $\{\tau\leq t\}$
and the last one is true since $T_n\land t$ and $\tau_t$ are two $\F_t$-measurable random variables and
\begin{align*}
\{T_n=\tau_t\}
&=\{T_n=\tau_t<t\}\cup \{T_n=\tau_t=t\}\\
&=\big\{\{T_n\land t=\tau_t\}\cap\{\tau_t<t\}\big\}\cup\big\{\{T_n=t\}\cap\{\tau_t=t\}\big\}.
\end{align*}
The dual optional projection of a thin time satisfies
\begin{align*}
\id_{\{T_n =\tau_t\}}A^o_t&=\sum_{k=1}^\infty \id_{\{T_n=\tau_t,\, T_k\leq t\}} z^k_{T_k}
=\id_{\{T_n=\tau_t\}} z^n_{T_n},
\end{align*}
where the second equality is due to the fact that for $n\neq k$ we have
\begin{align*}
\id_{\{T_n=\tau_t,\, T_k\leq t\}} z^k_{T_k}
&=\id_{\{T_n=\tau_t\}} \E(\id_{\{\tau=T_k\leq t\}}|{T_k})\\
&=\id_{\{T_n=\tau_t=T_k\}} \E(\id_{\{\tau=T_k\leq t\}}|{T_k})
=0
\end{align*}
since $T_n$ and $T_k$ have disjoint graphs and $\tau$ is an honest time.
Combining the two previous points, we conclude that $1-m_t=1-Z_t-A^o_t=z^n_t-z^n_{T_n}$ on the set $\{T_n=\tau_t\}$.\\
(b) Again using properties of the random variable $\tau_t$ we derive
\begin{align*}
\id_{\{T_n<t\}} z^n_t&=\P(\tau=T_n=\tau_t<t|\F_t)
=\id_{\{T_n=\tau_t<t\}} (1-\wt Z_t),\\
\id_{\{T_n<t\}} z^n_{t-}&=\P(\tau=T_n=\tau_t<t|\F_{t-})
=\id_{\{T_n=\tau_t<t\}} (1-Z_{t-}).
\end{align*}
Then, Proposition \ref{positive} (c) completes the proof.
\end{proof}

For progressive enlargement with an honest time, the hypothesis
(${\mathcal H^\prime}$) is satisfied, and the following decomposition is
given in \cite[Theorem (5,10)]{j}.
Let $M$ be an $\ff$-local martingale. Then, there exists an $\ff^\tau$-local martingale $\wh M$ such that:
\begin{equation}
 \label{aaafter}
M_t = \wh M_t+\int_0^{t\land \tau} \frac{1}{Z_{s-}}d\cro{M,m}_s-\int_0^{t} \id_{\{s>\tau\}} \frac{1}{1-Z_{s-}}d\cro{M, m}_s.
\end{equation}

\begin{remark}
For a thin honest time $\tau$, the two decomposition formulas, first given in  Theorem \ref{decomposition1} and second given in \eqref{aaafter}, coincide.
It is enough to show that
\begin{align*}
\int_0^t \id_{\{s>\tau\}} \frac{1}{1-Z_{s-}} d\cro{X,1-m}_s
&=\sum_{n=1}^\infty \id_{C_n}\int_0^t \id_{\{s>T_n\}} \frac{1}{z^n_{s-}} d\cro{X,z^n}_s.
\end{align*}
This is a simple consequence of the set inclusion $\{\tau< s\}\cap\{\tau=T_n\} \subset \{T_n=\tau_s\leq s\}$ and Lemma \ref{thinhonest} (a):
\begin{align*}
\int_0^t \id_{\{s>\tau\}} \frac{1}{1-Z_{s-}}  d\cro{X,1-m}_s
&=\sum_{n=1}^\infty \int_0^t \id_{\{s>\tau\}\cap\{\tau=T_n\}} \frac{1}{1-Z_{s-}}  d\cro{X,1-m}_s\\
&=\sum_{n=1}^\infty \int_0^t \id_{\{s>\tau\}\cap\{\tau=T_n\}} \frac{1}{z^n_{s-}} d\cro{X,z^n }_s\\
&=\sum_{n=1}^\infty \id_{C_n}\int_0^t \id_{\{s>T_n\}} \frac{1}{z^n_{s-}} d\cro{X,z^n}_s.
\end{align*}
\end{remark}

\subsection{Jumping filtration}

In this subsection we develop the relationship between jumping filtration and thin honest times.
Let us first recall the definition of a jumping filtration and the main result obtained in Jacod and Skorokhod \cite{jac}.

\begin{definition}
\label{jumping_def}
A filtration $\ff$ is called a jumping filtration if there exists a localizing sequence $(\theta_n)_{n\geq 0}$, i.e., a sequence of stopping times increasing a.s. to $\infty$, with $\theta_0=0$ and such that, for all $n$ and $t>0$,
the $\sigma$-fields $\F_t$ and $\F_{\theta_n}$ coincide up to null sets on $\{\theta_n\leq t < \theta_{n+1}\}$.\\
The sequence $(\theta_n)_n$ is then called a jumping sequence.
\end{definition}

There exists an important alternative characterization of jumping filtration in terms of martingale's variation (\cite[Theorem 1]{jac}).

\begin{theorem}
\label{jumping_theo}
The two following conditions are equivalent:\\
(a) a filtration $\ff$ is a jumping filtration;\\
(b) all martingales in the filtration $\ff$ are a.s.\! of locally finite variation.
\end{theorem}

We investigate relationship between jumping filtration and honest times.
We show that there does not exist thick  honest time in a jumping filtration and that there exists a thick  honest time in a filtration which admits a non-constant continuous martingale (in particular such a filtration is not a jumping filtration).

\begin{theorem}
\label{jump_honest}
The following assertions hold.\\
(a) If $\ff$ is a jumping filtration, then all $\ff$-honest times are thin.\\
(b) If all $\ff$-honest times are thin, then all non-constant $\ff$-local martingales are purely discontinuous.
\end{theorem}

\begin{proof}
(a) Let $\tau$ be an honest time.
Then, take the same process $\alpha$ as in the proof of Proposition \ref{honest_alt},
i.e., $\alpha$ is an increasing, c\`adl\`ag, adapted process such that $\alpha_t=\tau$ on $\{\tau\leq t\}$ and $\tau=\sup\{t : \alpha_t=t\}$.
Let us define the partition $(C_n)_{n=0}^\infty$ such that
\begin{equation*}
C_n=\{\theta_{n-1}\leq \tau <\theta_{n}\}
\end{equation*}
 for $n\geq 1$ and $C_0=\{\tau=\infty\}$ with $(\theta_n)_{n\geq 0}$ being a jumping sequence for the jumping filtration $\ff$.
On each $C_n$ with $n\geq 1$ we have
$$\tau=T_n:=\inf\{t\geq \theta_{n-1} : t=\alpha_{\theta_n-}\}.$$
From the jumping filtration property, we know that $\alpha_{\theta_n-}$ is $\F_{\theta_{n-1}}$-measurable so each $T_n$ is a stopping time
and $\graph{\tau} \subset \bigcup_{n=1}^\infty \graph{T_n}$ which shows that the honest time $\tau$ is a thin time.

(b) The proof by contradiction is based on \cite[Exercise (1.26) p.235]{revuz}.
Assume that $M$ is a non-constant continuous $\ff$-local martingale with $M_0=0$.
Define the $\ff$-stopping time $S_1=\inf\{t>0: \cro{M}_t=1\}$.
Then, define the $\ff$-honest time
\[\tau:=\sup\left\{t\leq S_1: M_t=0 \right \}.\]
Since $M$ is continuous, $\tau$ is not equal to infinity with strictly positive probability.
We now show that $\tau$ is an $\ff$-thick honest time.
Let us denote $\mathcal Z(\omega):=\{t: M_t(\omega)=0\}$. The set $\mathcal Z(\omega)$ is closed and $\mathcal Z^c(\omega)$ is the union of countably many open intervals. We call $G(\omega)$ the set of left ends of these open intervals.
In what follows we show that for any $\ff$-stopping time $T$ we have $\P(T\in G)=0$.
Define the $\ff$-stopping time
$$D_T:=\inf\{t>T: M_t=0\}$$
and note that
$$\{T\in G\}=\{M_T=0\}\cap\{T<D_T\}\in \F_T.$$
Assume $\P(T\in G)=p>0$.
Then the process
$$Y_t=\id_{\{T\in G\}}|M_{T+t}|\id_{\{0\leq t \leq D_T-T\}}$$
is an $(\F_{T+t})_{t\geq 0}$-martingale. Indeed, for $s\leq t$ we have
\begin{align*}
\E(Y_t&|\F_{T+s})=\id_{\{T\in G\}}\sgn{M_{T+t}}\E(M_{T+t}\id_{\{t \leq D_T-T\}}|\F_{T+s})\\
=&\id_{\{T\in G\}}\sgn{M_{T+t}}\\
&\left [M_{T+s}\id_{\{s \leq D_T-T\}}
-\E(\id_{\{s \leq D_T-T\}}\id_{\{t > D_T-T\}}\E(M_{T+t}|\F_{D_T})|\F_{T+s})\right]\\
=&Y_s-\id_{\{T\in G\}}\sgn{M_{T+t}}\E(\id_{\{s \leq D_T-T\}}\id_{\{t > D_T-T\}}M_{D_T}|\F_{T+s})\\
=&Y_s
\end{align*}
where we have used the martingale property of $M$ and $M_{D_T}=0$.
Moreover $Y_0=0$ and there exists $\varepsilon >0$ such that
$$\P(M_T=0, D_T-T>\varepsilon)\geq \frac{p}{2}>0.$$
Since $Y_\varepsilon=\id_{\{M_T=0\}}\id_{\{D_T-T\geq \varepsilon\}}|M_{T+\varepsilon}|\geq 0$
and $\P(Y_\varepsilon>0)>0$, we have $\E(Y_\varepsilon)>0=Y_0$. So, $\P(T\in G)=0$.
Finally, as $\tau\in G$ a.s. we conclude that $\tau$ is a thick  honest time.
\end{proof}

Finally we give two examples of thick  honest times originating from purely discontinuous semimartingales of infinite variation. In the first Example \ref{azema_ex}, we study the case of Az\'ema's martingale  (see \cite[IV.8 p.232-237]{protter}). In the second Example \ref{kostas_ex}, we recall Example 2.1 from \cite{Khonest} on \emph{Maximum of downwards drifting spectrally negative L\'evy processes with paths of infinite variation}.

\begin{example}
\label{azema_ex}
Let $B$ be a Brownian motion and $\ff$ its natural filtration.
Define the process
$$g_t:=\sup\{s\leq t: B_s=0\}.$$
The process
$$\mu_t:=\sgn{B_t}\sqrt{t-g_t}$$
is a martingale with respect to the filtration $\gg:=(\F_{g_t+})_{t\geq 0}$ and is called the Az\'ema martingale.
Then, the random time
$$\tau:=\sup\{t\leq 1: \mu_t=0\}$$
is clearly a $\gg$-honest time.
Note that $\tau=\tau^B:=\sup\{t\leq 1: B_t=0\}$ and $\tau^B$ is an $\ff$-thick honest time (see in \cite[Table 1$\alpha$ 1), p.32]{mansuyyor} that $\tau^B$ has continuous $\ff$-dual optional projection).
Thus, since $\gg\subset\ff$, $\tau$ is a $\gg$-thick honest time .
\end{example}

\begin{example}
\label{kostas_ex}
Let $X$ be a L\'evy process with characteristic triplet $(\alpha, \sigma^2=0, \nu)$
satisfying $\nu((0,\infty))=0$, $\alpha+\int_{-\infty}^{-1}x\nu(dx)<0$ and $\int_{-1}^0 |x|\nu(dx)=\infty$.
Then, $\rho=\sup\{t: X_{t-}=X^*_{t-}\}$ with $X^*_t=\sup_{s\leq t}X_s$ is a thick  honest time as shown in \cite[Section 2.1]{Khonest}.
\end{example}

\appendix

\section{Definitions of projections}
\label{projections}

We collect here the definitions of the key tools we have used along the paper. Projections and dual projections onto the reference filtration $\ff$ play an important role in the theory of enlargement of filtrations.
First we recall the definition of optional and predictable projections, see \cite[Theorems 5.1 and 5.2]{chinois} and \cite[p.264-265]{3M}.

\begin{definition}
Let $X$ be a measurable bounded (or positive) process.
The optional projection of $X$ is the unique optional process $\ooo{X}$ such that for every stopping time $T$ we have
$$\E\left[X_T\id_{\{T<\infty\}}|\F_T\right]=\ooo{X}_T\id_{\{T<\infty\}} \quad \textrm{a.s.}.$$
The predictable projection of $X$ is the unique predictable process $\p{X}$ such that for every predictable stopping time $T$ we have
$$\E\left[X_T\id_{\{T<\infty\}}|\F_{T-}\right]=\p{X}_T\id_{\{T<\infty\}} \quad \textrm{a.s.}.$$
\end{definition}

For definition of dual optional projection and dual predictable projection see \cite[p.265]{3M}, \cite[Chapter 3 Section 5]{protter}, \cite[Chapter 6 Paragraph 73 p.148]{dell58}, \cite[Sections 5.18, 5.19]{chinois}.
We point out that the convention we use here allows a jump at $0$, where for a finite variation process $V$ we assume that $V_{0-}=0$.

\begin{definition}
\label{dual}
(a) Let $V$ be a c\`adl\`ag pre-locally integrable variation process (not necessary adapted).
The dual optional projection of $V$ is the unique optional process $\oo{V}$ such that for every optional process $H$ we have
$$\E\left[\int_{[0,\infty)} H_sdV_s\right]=\E\left[\int_{[0,\infty)} H_sd\oo{V}_s\right].$$
In particular, $V^o_0=\E[V_0|\F_0]$.

(b) Let $V$ be a c\`adl\`ag locally integrable variation process (not necessary adapted).
The dual predictable projection of $V$ is the unique predictable process $\pp{V}$ such that for every predictable process $H$ we have
$$\E\left[\int_{[0,\infty)} H_sdV_s\right]=\E\left[\int_{[0,\infty)} H_sd\pp{V}_s\right].$$
In particular, $V^p_0=\E[V_0|\F_0]$.
\end{definition}

{\section{Auxiliary results on honest times}
\label{honest_appendix} }

For reader's convenience we gather complementary results on honest times. They can be found in \cite{j} (see Lemma 5,1 and its proof there).

\begin{proposition}
\label{honest_alt}
(a) A random time $\tau$ is an $\ff$-honest time if and only if for every $t > 0$ there exists an $\F_{t-}$-measurable random variable $\tau_t$ such that $\tau=\tau_t$ on $\{\tau<t\}$.\\
(b) A random time $\tau$ is an $\ff$-honest time if and only if for every $t > 0$ there exists an $\F_t$-measurable random variable $\tau_t$ such that $\tau=\tau_t$ on $\{\tau\leq t\}$.
\end{proposition}

\begin{proof}
Sufficiency of both conditions is straightforward.

Using the notation from Definition \ref{defhonest} we introduce the process $\alpha^-$
as
$\alpha^-_t=\sup_{r\in Q, r<t}\tau_r$.
This definition implies that $\alpha^-$ is an increasing, left-continuous, adapted process such that $\alpha^-_t=\tau$
on $\{\tau<t\}$ thus the necessary condition in (a) is proven.

Let us denote by $\alpha$ the right-continuous version of $\alpha^-$, i.e., $\alpha_t=\alpha^-_{t+}$. Then, $\alpha$ is an increasing, c\`adl\`ag, adapted process such that $\alpha_t=\tau$ on $\{\tau\leq t\}$ and $\tau=\sup\{t : \alpha_t=t\}$ thus the necessary condition in (b) is proved.
\end{proof}

\begin{theorem}
\label{jeulin5,1}
Let $\tau$ be a random time.
Then, the following conditions are equivalent:\\
(a) $\tau$ is an honest time;\\
(b) there exists an optional set $\Gamma$ such that $\tau(\omega)=\sup\{t: (\omega, t)\in \Gamma\}$ on $\{\tau<\infty\}$;\\
(c) $\wt Z_\tau=1$ a.s. on $\{\tau<\infty\}$;\\
(d) $\tau=\sup\{t: \widetilde Z_t=1 \}$ a.s. on $\{\tau<\infty\}$.
\end{theorem}

\bibliographystyle{ACM} 
\bibliography{biblio}

\end{document}